\documentclass[reqno,oneside,10pt]{amsart}
\usepackage[width=16cm,centering]{geometry}
 \usepackage{fancyhdr,enumitem,amsmath,amssymb}
 \usepackage[colorlinks=true,citecolor=blue,menucolor=blue]{hyperref}
 
\usepackage{tikz}
 
 \title{Fixed product preserving mappings on Banach Algebras}
 \author{Hayden Julius}
 
\newtheorem{theorem}{Theorem}

\newtheorem{lemma}[theorem]{Lemma}
\newtheorem{corollary}[theorem]{Corollary}
\newtheorem{proposition}[theorem]{Proposition}
\newtheorem{problem}[theorem]{Problem}
\newtheorem*{rem}{Remark}
\numberwithin{theorem}{section}

\newcommand{\C}{\mathbb{C}}
\newcommand{\ve}{\varepsilon}
\newcommand{\im}{\textup{im}}
\newcommand{\rank}{\textup{rank}}

\newcommand{\norm}[1]{\lVert#1\rVert}
\numberwithin{equation}{section} 

\newtheoremstyle{namedp}{}{}{\itshape}{}{\bfseries}{.}{.5em}{\thmnote{Problem #3}}
\theoremstyle{namedp}
\newtheorem*{namedpproblem}{Problem}

\newtheoremstyle{named}{}{}{\itshape}{}{\bfseries}{.}{.5em}{\thmnote{Theorem #3}}
\theoremstyle{named}
\newtheorem*{namedtheorem}{Theorem}


\address{Department of Mathematics and Statistics, 
		Youngstown State University, Youngstown, OH 44555 U.S.A.} 
\email{hjulius@ysu.edu}	  

\begin{document}
\begin{abstract}
In this paper, we describe linear maps between complex Banach algebras that preserve products equal to fixed elements. This generalizes some important special cases where the fixed elements are the zero or identity element. First we show that if such map preserves products equal to a finite-rank operator, then it must also preserve the zero product. In several instances, this is enough to show that a product preserving map must be a scalar multiple of an algebra homomorphism. Second, we explore a more general problem concerning the existence of product preserving maps and the relationship between the fixed elements. Lastly, motivated by Kaplansky's problem on invertibility preservers, we show that maps preserving products equal to fixed invertible elements are either homomorphisms or antihomomorphisms multiplied on the left by a fixed element.
\end{abstract}

\maketitle
\section{Introduction}
This paper is primarily concerned with the existence and description of linear mappings between algebras taking products equal to one fixed element to products equal to another fixed element, in the sense of the following problem.
\begin{namedpproblem}[P]\label{mainprob}
Let $\mathcal A$ and $\mathcal B$ be algebras and $\phi: \mathcal A \to \mathcal B$ a bijective linear map such that $\phi(a)\phi(b)=d$ whenever $ab = c$, where $c \in \mathcal A$ and $d \in \mathcal B$ are fixed. When is $\phi$ a scalar multiple of a ring homomorphism?
\end{namedpproblem}

For brevity, we refer to a map $\phi$ satisfying the hypothesis of Problem P as a \textit{fixed product preserving mapping}. We may also refer to $\phi$ as \textit{preserving fixed products at $c$} to be more specific about the element of interest. It is also desirable to weaken the hypotheses on $\phi$ wherever possible. Occasionally below we can remove bijectivity, linearity, etc. 

Perhaps the most well-known version of this problem takes $c =d=0$; that is, $\phi$ satisfies $\phi(a)\phi(b) = 0$ whenever $ab = 0$, for all $a,b \in \mathcal A$. We say then that $\phi$ \textit{preserves products equal to zero} or \textit{preserves the zero product}. Such maps also go by the name \textit{Lamperti operators, separating maps,} or \textit{disjointness preserving maps} on function algebras. A substantial number of papers are devoted to the zero product preserver problem. A small, representative sample for our purposes is \cite{chebotar,houcui,wong}. The most common conclusion is that the range of $\phi$ either has trivial multiplication, or $\phi$ is a scalar multiple of a homomorphism. 

The corresponding notion of a \textit{zero product determined} (zpd) \textit{algebra} builds on this idea; these are algebras for which every bilinear form vanishing on zero products must be implemented by a single linear transformation. The subject arose out of ideas defined in \cite{alaminos,bresarzpd} and many subsequent publications; the most recent and summative source on zpd algebras is Bre\v sar's book \cite{bresarzpdbook}. 
The book \cite{bresarzpdbook} contains an extensive list of references for zpd algebras (in both analytic and algebraic contexts) and zero product preservers.





Problem P for nonzero elements is beginning to be explored in greater depth; for instance, it is solved completely on $M_n(\C)$, the full algebra of $n \times n$ matrices with complex entries. In the recent paper \cite{catalanojulius}, the author and his collaborator Catalano addressed this exact problem for any fixed matrices $C$ and $D$ of the same rank (note: we use capital letters to denote operators and matrices). A number of papers built up specific cases, such as \cite{catalano,catalanochanglee,catalanohsukapalko}). See also \cite{zhu}.  Costara \cite{costaraproducts} found a topological proof showing that the existence of a fixed product preserving map automatically implies that $\rank(C) = \rank(D)$. This observation directly inspired some of the results in Section \ref{connect}. 

Fixed product preserving mappings also arise naturally for other operations, such as the Lie product $[a,b] = ab-ba$ and the Jordan product $a \circ b = ab+ba$, which are of principal interest. Some relevant problems for the Lie product can be found in \cite{GJV,julius2,omladicsemrl,semrl2}. For the Jordan product, see \cite{catalanohsukapalko,cklz}; most notably, the authors in \cite{catalanohsukapalko} obtained a complete description for maps preserving equal fixed Jordan products on $M_n(\C)$. However, the Lie case can be pathological (such as in \cite{GJV}). General approaches (even for $M_n(\C)$) for the Lie product are not known and seems to be a challenge. 

It is the aim of this paper to extend some of the known solutions to Problem P to algebras of operators on infinite-dimensional spaces and general Banach algebras. There are three sections. In Section \ref{finite}, we take $C$ to be a finite-rank operator in $\mathcal B(X)$ and find a reduction to the zero product, which answers Problem P in the affirmative when the range of $\phi$ is prime (recall that a ring is called \textit{prime} if the product of any two nonzero ideals is nonzero). 
\begin{namedtheorem}[2.2]
Let $\mathcal A$ be a prime unital algebra and $X$ be an infinite-dimensional Banach space. If $\phi: \mathcal B(X) \to \mathcal A$ is a bijective linear map such that $\phi(A)\phi(B) = D$ whenever $AB = C$, where $C$ is a finite-rank operator, then $\phi(I) = \alpha$ is an invertible element lying in the center of $\mathcal A$ and there exists an isomorphism $\Phi: \mathcal B(X) \to \mathcal A$ such that $\phi(T) = \alpha \Phi(T)$ for all $T \in \mathcal B(X)$.
\end{namedtheorem}
In Section \ref{connect}, a new question is proposed. Suppose that there exists a map preserving products at $c$. How are $c$ and $d$ related? Some general criteria are provided and our most conclusive results hold in Banach algebras whose group of invertible elements is dense in the norm topology. 

\begin{namedtheorem}[3.3]
Let $\mathcal A$ be a unital Banach algebra whose group of invertible elements is dense in the norm topology. Fix two elements $c,d \in \mathcal A$ such that $\phi: \mathcal A \to \mathcal A$ is a bijective continuous linear map satisfying $\phi(a)\phi(b) = d$ whenever $ab = c$. Then 
\begin{enumerate}\upshape
\item \textit{$c = 0$ if and only if $d = 0$, }
\item \textit{$c$ is invertible if and only if $d$ is invertible.}
\end{enumerate}
\end{namedtheorem}

Some new approaches would be required to address perhaps the most interesting case of fixed product preservers on $\mathcal B(X)$. In Section \ref{inv}, motivated by Kaplansky's problem on invertibility preserving mappings (see [Section 9, \cite{kaplanskyinv}] and [Section 0.1, \cite{molnar}] for discussion), we demonstrate that mappings preserving fixed products at an invertible element, under mild conditions, must be a homomorphism or antihomomorphism up to multiplication by a fixed element. 

\begin{namedtheorem}[4.1]
Let $\mathcal A$ and $\mathcal B$ be unital Banach algebras and let $\phi: \mathcal A \to \mathcal B$ be a bijective linear map such that $\phi(a)\phi(b) = d$ whenever $ab = c$, where $d$ is invertible. Then
\begin{enumerate}\normalfont
\item \textit{$\phi(x^{-1}) = z\phi(x)^{-1}z$ for all invertible $x \in \mathcal A$, where $z = \phi(1)$,} 
\item \textit{$z^{-1}\phi$ strongly preserves invertibility, therefore}
\item \textit{$z^{-1}\phi$ a Jordan homomorphism, and}
\item \textit{if $\mathcal B$ is prime, then $z^{-1}\phi$ is either a homomorphism or antihomomorphism.}
\end{enumerate}
\end{namedtheorem}

\newcommand{\ran}{\textup{im}}

Lastly, we fix some notation. $X$ always denotes a complex Banach space, $X'$ denotes its dual of bounded linear functionals, and $\mathcal B(X)$ denotes the algebra of bounded linear operators on $X$ with the identity operator denoted $I$. Elements of $\mathcal B(X)$ are simply called \textit{operators}.  Operators will be denoted with upper-case letters and vectors with lower-case letters. The range of an operator $T$ is denoted by $\ran(T)$. 
Only unital associative algebras are considered in this paper. By a \textit{Jordan homomorphism} $\psi: \mathcal A \to \mathcal B$, where $\mathcal A$ and $\mathcal B$ are algebras, we mean a linear map such that $\psi(a^2) = \psi(a)^2$ for all $a \in \mathcal A$.

\section{Fixed product preserving maps at finite-rank operators}\label{finite}

Let $X$ be a Banach space. A rank-one operator in $\mathcal B(X)$ can be written as $x \otimes f$, where $\im(x \otimes f) = \C x$, and $f\in X'$ is a nonzero bounded linear functional. Vector multiplication is defined by $(x\otimes f)(y) = f(y)x$ for all $y \in X$. Note that $x\otimes f$ is a projection if and only if $f(x) = 1$. A finite-rank operator is a sum of rank-one operators, and a finite-rank projection is therefore a sum of rank-one projections.

Taking $C$ to be a finite rank operator in Problem P, our first lemma provides a reduction to the usual zero product preserver problem. In fact, the map $\phi$ need only be additive for this reduction to take place. Moreover, since continuity is not needed either, the range of $\phi$ can be an arbitrary (not necessarily Banach) algebra. The following proof relies only on the elementary theory of vector spaces and the Hahn-Banach theorem to ensure the existence of certain bounded linear functionals.

\begin{proposition}\label{frank} Let $\mathcal A$ be an arbitrary associative algebra and $X$ an infinite-dimensional Banach space. If $\phi: \mathcal B(X) \to \mathcal A$ is an additive map preserving products at $C$, a rank-$k$ operator, then $\phi$ preserves the zero product. 
\end{proposition}

\begin{proof} Write $$C = v_1\otimes f_1 + \cdots + v_k \otimes f_k$$ for linearly independent vectors $\{v_1,\dotsc,v_k\}$ in $X$ and linearly independent bounded linear functionals $\{f_1,\dotsc,f_k\}$ in $X'$. 
Let $P$ and $Q$ be arbitrary nonzero operators such that $QP = 0$. 

First, suppose the $\ran(P)$ is a closed subspace of $\ker(Q)$ such that $\dim \ker(Q) / \ran P \geq k$ (in particular, $\ran(P)$ is not dense in $\ker(Q)$). 
Let $W$ be a $k$-dimensional subspace of $\ker(Q)$ disjoint from $\ran(P)$ and let $\{w_1,\dotsc,w_k\}$ be a basis of $W$. Define the operator $$A = v_1\otimes g_1 + \dotsc+ v_k\otimes g_k,$$ where the bounded linear functionals $g_i$ satisfy
$$g_i(w_j) = \delta_{ij}, \quad j = 1,\dotsc,k, \quad \text{ and } \quad g_i(\ran(P)) = 0$$
for all $i =1,\dotsc,k$. By construction, $AP = 0$. Define the operator $$B = w_1\otimes f_1 + \cdots + w_k\otimes f_k.$$ Since the $w_i$ are elements of $\ker(Q)$, by construction we have $QB = 0$ and it is easy to see that $AB = C$. Hence we have four equations:
$$AB = C, \qquad \phantom{\text{ and }} \qquad  (A+Q)B = C,$$
$$A(P+B) = C, \qquad \text{ and } \qquad (A+Q)(P+B) = C.$$

From additivity, the two equations $\phi(A)\phi(B) = D$ and $\phi(A+Q)\phi(B) = D$ imply that $\phi(Q)\phi(B) = 0$, and $\phi(A)\phi(P) = 0$ is done analogously. Finally, expanding the parentheses of $\phi(A+Q)\phi(P+B) = D$, along with the previous observations, implies that $\phi(Q)\phi(P) = 0$. 

If the dimension of $\ker(Q)/\ran(P)$ is less than $k$, it is not as simple to find such factorizations of $C$ including $Q$ and $P$. However, we can always reduce to the above case by decomposing $Q$ or $P$ with respect to a finite-rank projection to obtain quotients of dimension at least $k$ as follows.

Suppose that both $\ker(Q)$ and $\ran(P)$ are infinite-dimensional
with $\dim \ker(Q)/\ran(P) < k$. We do not assume that $\ran(P)$ is closed in $\ker(Q)$. Let $E$ be a finite rank projection onto an $n$-dimensional subspace of $\ran(P)$. 
Let $I_{\ker(Q)}$ denote the identity operator on $\ker(Q)$. Since $\ran(E)$ is finite-dimensional and $\ker(Q)$ is a normed space, decompose $\ker(Q) = \ran(E) \oplus \ran(I_{\ker(Q)} - E)$ as both an algebraic and topological direct sum; that is, both $\ran(E)$ and $\ran(I_{\ker(Q)}-E)$ are closed and complemented subspaces of $\ker(Q)$. It follows from the original argument that $EP$ and $(I_{\ker(Q)}-E)P$; hence $\phi(Q)\phi(EP) = 0$ and $\phi(Q)\phi((I_{\ker(Q)}-E)P) = 0$. By additivity,
$$\phi(Q)\phi(P) = \phi(Q)\phi(EP + (I_{\ker(Q)} - E)P) = 0.$$




On the other hand, if $\ker(Q)$ is finite-dimensional (therefore $P$ is finite-rank) and $\dim \ker(Q) / \ran(P) < k$, similarly decompose $Q$ as $EQ + (I-E)Q$ for finite-rank projections, $E$, where $I$ is the identity operator on $X$. Now, $\ker(EQ)$ has an infinite-dimensional kernel containing the finite-dimensional closed subspace $\ran(P)$, so $\phi(EQ)\phi(P) = 0$ by previous remarks. This does not depend on the rank of $E$. However, if one chooses $E$ to project onto $n$ linearly independent vectors in $\ran(Q)$, then $\ker ((I-E)Q)$ has dimension $n + \dim \ker(Q)$. But $n$ can be arbitrary (as $\ran(Q)$ is infinite-dimensional) and $\ran(P)$ is a finite-dimensional closed subspace of $\ker((I-E)Q)$. Select $n$ large enough so that $\dim \ker((I-E)Q) / \ran(P) \geq k$. Thus $\phi((I-E)Q)\phi(P) = 0$ and by additivity, we get 
$$\phi(Q)\phi(P) = \phi(EQ + (I-E)Q) \phi(P) = 0$$
for some appropriately chosen projection $E$. Therefore $\phi$ preserves the zero product.
\end{proof}

Let $H_1$ and $H_2$ be infinite-dimensional complex Hilbert spaces and $\phi:\mathcal B(H_1) \to \mathcal B(H_2)$ a bijective additive map preserving products at $C$, a finite-rank operator. It follows from Proposition \ref{frank} and [Corollary 2.8, \cite{chebotar}] that $\phi(T) = \alpha UTU^{-1}$, where $\alpha$ is a nonzero scalar and $U$ is an invertible bounded linear operator from $H_1$ onto $H_2$. This solves Problem P for finite-rank operators on Hilbert spaces. 

Along these lines, there is a similar result for arbitrary Banach spaces $X$, as long as the target space of $\phi$ is a \textit{prime} unital algebra. As mentioned in the introduction, a ring is prime if the product of two nonzero ideals is always nonzero (in particular, $\mathcal B(X)$ is prime). 

\begin{theorem}\label{frankt}
Let $\mathcal A$ be a prime unital algebra and $X$ be an infinite-dimensional Banach space. If $\phi: \mathcal B(X) \to \mathcal A$ is a bijective linear map such that $\phi(A)\phi(B) = D$ whenever $AB = C$, where $C$ is a finite-rank operator, then $\phi(I) = \alpha$ is an invertible element lying in the center of $\mathcal A$ and there exists an isomorphism $\Phi: \mathcal B(X) \to \mathcal A$ such that $\phi(T) = \alpha \Phi(T)$ for all $T \in \mathcal B(X)$.
\end{theorem}

\begin{proof}
By Proposition \ref{frank}, $\phi$ preserves the zero product. The rest follows from [Theorem 7.8, \cite{bresarzpdbook}]. 
\end{proof}

A few notes are in order. $\mathcal B(X)$ is not always zpd but contains plenty of noncentral idempotents, a necessary hypothesis of [Theorem 7.8, \cite{bresarzpdbook}], which 
circumvents the zpd assumption and therefore does not require continuity. Taking $\mathcal A = \mathcal B(X)$, then we conclude that $\phi$ is a scalar multiple of an automorphism and thus of the form $\phi(T) = \alpha UTU^{-1}$ for some invertible operator $U$ (see Eidelheit \cite{eidelheit} for the proof that every automorphism of $\mathcal B(X)$ is inner). Furthermore, it is easily seen that $D = \alpha \phi(C)$. To emphasize, we did not assume anything about $D$ itself. Lastly, Theorem \ref{frankt} holds for finite-dimensional spaces $X$ provided that $C$ is noninvertible.
When $C$ is invertible, the map $\phi$ can behave like an antiautomorphism. See the remark following Theorem \ref{zaz} below. 

Let $\mathcal F(X)$ denote the space of finite-rank operators in $\mathcal B(X)$. Any unital subalgebra of $\mathcal B(X)$ containing $\mathcal F(X)$ also has the necessary operators to recreate the argument in Proposition \ref{frank} and Theorem \ref{frankt}.

\begin{corollary}\label{frankcor}
Let $\mathcal S$ be a unital subalgebra of $\mathcal B(X)$ containing $\mathcal F(X)$ and let $\mathcal A$ be a prime unital algebra. If $\phi: \mathcal S \to \mathcal A$ is a bijective linear map such that $\phi(A)\phi(B) = D$ whenever $AB = C$, where $C$ is a finite-rank operator, then $\phi$ is a scalar multiple of a homomorphism. 
\end{corollary}

\section{Connections between the fixed elements of Problem P}\label{connect}
\newcommand{\ann}{\textup{ann}}
The most common type of fixed product preserver problems take $c= d$ or assume that $\phi(c) = d$. We do not wish to assume this for Problem P. In all known examples, the existence of a fixed product preserving map automatically implies some relationship between $c$ and $d$. Most notably, in the finite-dimensional case of $M_n(\C)$, Costara \cite{costaraproducts} proved that $\rank(C) = \rank(D)$ is automatic and need not be assumed. Theorem \ref{frankt} above also draws a similar conclusion that $C$ finite-rank implies that $D$ is finite-rank. For this section, the question of the forced relationship between $c$ and $d$ imposed by a product preserving mapping is explored in greater generality. 

Let $\textup{ann}_\ell(c) = \{x\in \mathcal A: xc = 0\}$ denote the subspace of left annihilators of $c$ and $\textup{ann}_r(c) = \{x \in \mathcal A : cx = 0\}$ denote the subspace of right annihilators of $c$. It is easy to see that any annihilator of $c$ must be taken to an annihilator of $d$ in the image of a fixed product preserving map. 

\begin{proposition}\label{ann}
Let $\mathcal A$ and $\mathcal B$ be unital algebras. Fix two elements $c \in \mathcal A$ and $d \in \mathcal B$ and let $\phi: \mathcal A \to \mathcal B$ be a linear map such that $\phi(a)\phi(b) = d$ whenever $ab = c$. Then
\begin{enumerate}\upshape
\item \textit{$\phi(\textup{ann}_\ell(c)) \subseteq \textup{ann}_\ell(d)$ and $\phi(\textup{ann}_r(c)) \subseteq \textup{ann}_r(d)$,}
\item \textit{if $\phi$ is injective and $d$ is left (resp. right) invertible, then $c$ is left (resp. right) invertible, and}
\item \textit{if $\phi$ is injective and $d$ is invertible, then $c$ is invertible.}
\end{enumerate}
\end{proposition}

\begin{proof}
Suppose $x \in \textup{ann}_\ell(c)$. Then $(1+x)c = c$ 
implies that $\phi(1+x)\phi(c) = d$. By linearity, it follows that $\phi(x)\phi(c) = 0$, and hence 
\[0 = \phi(x)\phi(c)\phi(1) = \phi(x)d.\]
Therefore $\phi(\textup{ann}_\ell(c)) \subseteq \textup{ann}_\ell(d)$. The case for right annihilators is analogous. This establishes (1). 

For (2), if $d$ is left-invertible (that is, $\ann_r(d) = 0$), then $\phi(\ann_r(c)) = 0$. If $\phi$ is injective, it follows that $\ann_r(c) = 0$ and $c$ is left-invertible. The case for right-invertibility is analogous. Statement (3) follows from (2). 
\end{proof}

Next, we rely on invertible elements. Recall that the group of invertible elements of a Banach algebra is open in the norm topology. For the remainder of this section, we also assume that the invertible elements are dense as well. These Banach algebras are necessarily Dedekind-finite as rings (see, for example, [Proposition 3.1, \cite{rieffel}]).
While restrictive, some connections between $c$ and $d$ can be found. 

\begin{lemma}\label{balls}
Let $\mathcal A$ be a Banach algebra whose group of invertible elements is dense in the norm topology and let $\phi: \mathcal A \to \mathcal A$ be a surjective continuous linear map. For every open ball $B_\ve \subseteq \mathcal A$ of radius $\ve>0$, there exists an invertible operator $s \in B_\ve$ such that $\phi(s)$ is invertible.
\end{lemma}

\begin{proof}
This is an immediate consequence of the open mapping theorem and density of invertible elements in $\mathcal A$. 
\end{proof}

\begin{theorem}\label{prop}
Let $\mathcal A$ be a unital Banach algebra whose group of invertible elements is dense in the norm topology. Fix two elements $c,d \in \mathcal A$ such that $\phi: \mathcal A \to \mathcal A$ is a bijective continuous linear map satisfying $\phi(a)\phi(b) = d$ whenever $ab = c$. Then 
\begin{enumerate}\upshape
\item \textit{$c = 0$ if and only if $d = 0$, }
\item \textit{$c$ is invertible if and only if $d$ is invertible.}
\end{enumerate}
\end{theorem}

\begin{proof}
If $c = 0$, then $\phi(0)\phi(0) = 0=d$. On the other hand, suppose that $d = 0$ and $c \neq 0$. Let $t_0$ be invertible. There exists an $\varepsilon > 0$ such that $\norm{t - t_0} < \ve$ implies that $t$ is invertible. Hence $\phi(t)\phi(t^{-1}c) = 0$ for all $t \in B_\ve(t_0)$ (here, $B_\ve(t_0)$ denotes the ball of radius $\ve$ centered at $t_0$). If $\phi$ is injective, we have that $\phi(t^{-1}c) \neq 0$, hence $\phi(t)$ is not invertible for all $t \in B_\ve(t_0)$. But $\phi$ is a homeomorphism, so $\phi(B_\ve(t_0))$ is an open set and apparently consists only of noninvertible elements. This contradicts density and so $c=0$.


To prove statement (2), note that the backward direction is precisely Proposition \ref{ann}(3). For the forward direction, suppose $c$ is invertible. By Lemma \ref{balls}, let $s$ be any invertible element such that $\phi(s)$ is also invertible. Let $\ve > 0$ be chosen suitably small so that both the ball $B_\ve(s)$ and its image $\phi(B_\ve(s))$ consists only of invertible elements. Then $\phi(t)$ is invertible for all $t \in B_\ve(s)$. 

Consider the bijective map $f(a) = a^{-1}c$ on the group of invertible elements. Hence $f(B_\ve(s))$ is another open set consisting entirely of invertible elements of the form $a^{-1}c$, where $\norm{a - s} < \ve$. Another application of Lemma \ref{balls} gives the existence of an element $t \in B_\ve(s)$ such that $\phi(t^{-1}c)$ is also invertible. Now $\phi(t)$ and $\phi(t^{-1}c)$ are both invertible and their product is $d$, thus $d$ is invertible. Injectivity is not required.\end{proof}

Unless $X$ is finite-dimensional, $\mathcal B(X)$ does not generally have a dense invertible group. Regrettably Theorem \ref{prop} therefore does not apply to perhaps the most interesting cases. Nevertheless, it seems unlikely that the conclusions of the theorem should fail.

\begin{problem}
Does Theorem \ref{prop} hold if $\mathcal A = \mathcal B(X)$, where $X$ is an infinite-dimensional Banach space?
\end{problem}

In addition, Costara shows something stronger with [Lemma 1, \cite{costaraproducts}] for the Banach algebra of $n \times n$ matrices $M_n(\C)$; in particular, that $d \neq 0$ implies that $\phi$ preserves invertibility. The proof assumes that there is an invertible element $t_0$ such that $\phi(t_0)$ is not invertible. The argument works just as well for arbitrary Banach algebras with dense invertible group until the ultimate step when Costara uses convenient matrix representations of noninvertible elements to contradict the assumption on $\phi(t_0)$. We suspect that a similar result will hold for Theorem 3.3. 

\begin{problem}
If $\phi$ satisfies the hypotheses of Theorem \ref{prop}, does $d \neq 0$ imply that $\phi$ preserves invertibility?
\end{problem}

Lastly we explore the bijectivity assumption. When $d \neq 0$, clearly $\phi(t) \neq 0$ for every invertible element $t \in \mathcal A$. 
Suppose only that $\phi$ is surjective. Must $\phi$ be injective? For simplicity we will present an example for real Banach algebras but the idea remains the same for the complex case. Let $\mathcal A = C[0,1]$ be the Banach algebra of continuous functions on the interval $[0,1]$ endowed with the sup norm. The invertible elements are the nonvanishing continuous functions and are dense. Consider the linear map $\phi: C[0,1] \to C[0,1]$ defined by $\phi(f)(x) = f(\tfrac x2)$. The map is a surjective and continuous algebra homomorphism and therefore preserves products at a fixed function $c \in C[0,1]$ (and preserves invertibility). However the kernel of $\phi$ contains \textit{any} continuous function vanishing on $[0,\tfrac12]$. This also shows that the injectivity assumption of Theorem \ref{prop}(1) is indispensable.

\section{Fixed product preservers at invertible elements}\label{inv}

The final theorem provides a complete description of maps that preserve products at $c$ where the second fixed element, $d$, is invertible. It turns out that preserving products at $c$ also strongly preserves invertibility, and continuity need not be assumed. Recall that a map $\psi: \mathcal A \to \mathcal B$ is said to \textit{strongly preserve invertibility} if $\psi(x^{-1}) = \psi(x)^{-1}$ for all invertible $x \in \mathcal A$. 

\begin{theorem}\label{zaz}
Let $\mathcal A$ and $\mathcal B$ be unital Banach algebras and let $\phi: \mathcal A \to \mathcal B$ be a bijective linear map such that $\phi(a)\phi(b) = d$ whenever $ab = c$, where $d$ is invertible. Then
\begin{enumerate}\normalfont
\item \textit{$\phi(x^{-1}) = z\phi(x)^{-1}z$ for all invertible $x \in \mathcal A$, where $z = \phi(1)$,} 
\item \textit{$z^{-1}\phi$ strongly preserves invertibility, therefore}
\item \textit{$z^{-1}\phi$ a Jordan homomorphism, and}
\item \textit{if $\mathcal B$ is prime, then $z^{-1}\phi$ is either a homomorphism or antihomomorphism.}
\end{enumerate}
\end{theorem}

\begin{proof}
By Proposition \ref{ann}(3), $c$ is invertible. Moreover, $\phi$ preserves invertibility. Indeed, for every invertible element $t \in \mathcal A$, then $\phi(t)\phi(t^{-1}c)d^{-1} = 1$ and $d^{-1}\phi(ct^{-1})\phi(t)=1$, so $\phi(t)$ is invertible. 
The remainder of the proof largely follows the proof that was carried out in finite dimensions (see \cite{catalanohsukapalko,catalanojulius}), since only two very general facts are needed:
\begin{enumerate}[label = $\bullet$]
\item the spectrum of $a$, denoted $\sigma(a)$, is compact for all $a \in \mathcal A$, and 
\item a special case of Hua's identity; namely, $(1-a)^{-1} = 1+(a^{-1}-1)^{-1}$, which holds in any ring whenever $a$ and $1-a$ are both invertible.
\end{enumerate}

Let $x$ be an element such that $1-x$ is invertible. 
Using the special case of Hua's identity, 
\begin{equation*}\begin{aligned}
d &= \phi(1-x)\phi((1-x)^{-1}c)\\
&= \phi(1-x)\phi(c+(x^{-1}-1)^{-1}c)\\
&= d+z\phi((x^{-1}-1)^{-1}c)- \phi(x)\phi(c) - \phi(x)\phi((x^{-1}-1)^{-1}c)
\end{aligned}\end{equation*}
simplifies to 
\begin{equation*}
0 = z\phi((x^{-1}-1)^{-1}c)- \phi(x)\phi(c) - \phi(x)\phi((x^{-1}-1)^{-1}c).
\end{equation*}
For any invertible $t$, we have $\phi(t^{-1}c) = \phi(t)^{-1}d$. Revisiting the previous equation, 
\begin{equation*}
0 = z\phi(x^{-1}-1)^{-1}d - \phi(x)\phi(c) - \phi(x)\phi(x^{-1}-1)^{-1}d
\end{equation*}
and since $d$ is invertible (along with $z^{-1} = \phi(c)d^{-1}$),
$$0 = z\phi(x^{-1}-1)^{-1} - \phi(x)z^{-1} - \phi(x)\phi(x^{-1}-1)^{-1}.$$
Multiply through by $\phi(x^{-1}-1)$ on the right to obtain 
$$0 = z - \phi(x)z^{-1}\phi(x^{-1}-1) - \phi(x).$$
Writing $\phi(x^{-1}-1) = \phi(x^{-1}) - z$ and some rearranging reveals that
$$\phi(x)z^{-1}\phi(x^{-1})=z,$$
or, equivalently,
\begin{equation}\label{ZZ}\phi(x^{-1}) = z\phi(x)^{-1}z.\end{equation}
This establishes the desired conclusion for those invertible elements $x \in \mathcal A$ such that $1 - x$ is invertible.

In the event that $1-x$ is not invertible (in other words, $1 \in \sigma(x)$), here we use compactness of the spectrum to select a nonzero 
$\lambda \in \C$ such that $1-\lambda x$ is invertible. Returning to equation (\ref{ZZ}) we have
$$\phi((\lambda x)^{-1}) = z\phi(\lambda x)^{-1}z$$
and using the linearity of $\phi$ to cancel the common scalar $\lambda^{-1}$ on both sides of the equation, we conclude that
$$\phi(x^{-1}) = z\phi(x)^{-1}z$$
as well. Hence $\phi(x^{-1}) = z\phi(x)^{-1}z$ for all invertible 
$x \in \mathcal A$. 

For (2), define $\psi(x) :=z^{-1}\phi(x)$ for all $x \in \mathcal A$. Then $\psi:\mathcal A \to \mathcal B$ is a unital surjective continuous linear map that strongly preserves invertibility since, for all invertible $x \in \mathcal A$, 
\begin{equation*}\begin{aligned}
\psi(x^{-1}) &= z^{-1}\phi(x^{-1})\\
&= z^{-1}\phi(x^{-1})z^{-1}z\\
&= \phi(x)^{-1}z\\
&= (z^{-1}\phi(x))^{-1}\\
&= \psi(x)^{-1}
\end{aligned}\end{equation*}
by Equation (\ref{ZZ}). By [Theorem 2.2, \cite{boudimbekhta}] or [Theorem 2.2, \cite{burgosmarquezgarciamoralescampoy}], we then have that $\psi$ is a unital Jordan homomorphism. This proves (3).

For (4), we recall [Theorem H, \cite{herstein2}]: any Jordan homomorphism onto a prime ring $\mathcal B$ (of characteristic different from 2 or 3, which is certainly the case) is either a homomorphism or antihomomorphism.  \end{proof}

\begin{rem}\upshape
In general, $z^{-1}\phi$ as in Theorem \ref{zaz}(4) can be an antihomomorphism and $z$ need not be central. Let $\phi: M_n(\C) \to M_n(\C)$ be defined by $\phi(T) = \alpha DUT^tU^{-1}$, where $\alpha \in \C$, $T^t$ is the transpose of $T$, and both $D$ and $U$ are invertible and satisfy the equation $\alpha^2UC^tU^{-1} = D^{-1}$ for an invertible matrix $C$. Then $\phi(A)\phi(B) = D$ whenever $AB = C$. Notice that $z = \phi(I) = \alpha D$. This example is from \cite{catalanojulius2,catalanojulius}. 
\end{rem}

\begin{rem}\upshape Strong invertibility preservers have been frequently studied (see \cite{chebotarkeleeshiao,EK,linwong,mbekhta}) with some authors relaxing the condition to $\psi(a)\psi(a^{-1}) = \psi(b)\psi(b^{-1})$ for all invertible $a,b$ in the domain of $\psi$. Indeed, if $z$ is central in the above, then it is easy to see that $\phi(x)\phi(x^{-1}) = \phi(y)\phi(y^{-1})$ for all invertible $x,y \in \mathcal A$. Descriptions for maps satisfying this hypotheses on unital Banach algebras can be found in \cite{burgosmarquezgarciamoralescampoy}, which resolves Problem P in the case $c = d = 1$. Maps satisfying this condition are important for operator-theoretic generalizations of Hua's theorem for division rings [Theorem 9.1.3, \cite{cohn}].
\end{rem}

\section*{Acknowledgments}

The author would like to thank the referees for providing valuable feedback to improve the presentation of the paper.

\bibliographystyle{plain}
\bibliography{bib.tex}
\end{document}